\documentclass[letterpaper, 10 pt, conference]{ieeeconf}  

\IEEEoverridecommandlockouts                              

\usepackage{mathptmx,mathtools}
\usepackage{times}

\usepackage{algorithm} 
\usepackage{cite}
\usepackage{algorithmic}
\usepackage[algo2e,ruled,vlined]{algorithm2e}
\usepackage{amsmath,amsthm,mathtools}
\usepackage{amssymb,soul} 
\usepackage{xcolor}
\usepackage{tikz}
\usepackage{wrapfig}
\usetikzlibrary{arrows}
\newcommand{\reals}{\ensuremath \mathbb R}
\newcommand{\diag}{\ensuremath \Lambda}
\newcommand{\SSIO}{ SSIO\ }
\newcommand{\trace}{\ensuremath \text{trace}}
\usepackage{comment}
\newcommand{\s}{\ensuremath \mathbf{s}}
\newcommand{\m}{\ensuremath \mathbf{m}}
\newcommand{\y}{\ensuremath \mathbf{y}}
\newcommand{\z}{\ensuremath \mathbf{z}}
\newtheorem{theorem}{Theorem}

\usepackage{caption,bm}
\usepackage{subcaption}

\title{\LARGE \bf
Design of Experiments with Imputable Feature  Data: An Entropy-Based Approach
}

\author{Raj K. Velicheti, Amber Srivastava and Srinivasa M. Salapaka 
\thanks{This work was supported by DOE Subaward (WPI) 10809-GR, and the UIUC-ZJUI Center for Adaptive, Resilient Cyber-Physical Manufacturing Networks. The  authors  are  with the Coordinated  Science  Laboratory, University of Illinois at Urbana-Champaign,  IL,  61801  USA. {\tt\small $\{$rkv4,asrvstv6,salapaka$\}$@illinois.edu }}%

}

\begin{document}

\maketitle
\thispagestyle{empty}
\pagestyle{empty}

\begin{abstract}Tactical selection of experiments to estimate an underlying model is an innate task across various fields. Since each experiment has costs associated with it, selecting statistically significant experiments becomes necessary. Classic linear experimental design deals with experiment selection so as to minimize (functions of) variance in estimation of regression parameter. Typically, standard algorithms for solving this problem assume that data associated with each experiment is fully known. This isn't often true since missing data is a common problem. For instance, remote sensors often miss data due to poor connection. Hence experiment selection under such scenarios is a widespread but challenging task. Though decoupling the tasks and using standard \textit{data imputation} methods like matrix completion followed by experiment selection might seem a way forward, they perform sub-optimally since the tasks are naturally interdependent. Standard design of experiments is an NP hard problem, and the additional objective of imputing for missing data amplifies the computational complexity. In this paper, we propose a maximum-entropy-principle based framework that simultaneously addresses the problem of design of experiments as well as the imputation of missing data. Our algorithm exploits homotopy from a suitably chosen convex function to the non-convex cost function; hence avoiding poor local minima. Further, our proposed framework is flexible to incorporate additional application specific constraints. Simulations on various datasets show improvement in the cost value by over $60\%$ in comparison to benchmark algorithms applied sequentially to the imputation and experiment selection problems.
\end{abstract}

\section{INTRODUCTION}\label{section: introduction}
  
Parametric models are pivotal for analysis, design, control, and optimization in many studies and applications such as textile manufacturing\cite{durakovic2013continuous}, food \cite{yu2018design}, energy\cite{schlueter2018linking} and pharmaceutical\cite{paulo2017design} industries. Often, the unknown parameters in these models are estimated using experimental data. However, in many of these applications, running experiments are expensive and time-consuming. Therefore, it is very important to carefully select few experiments to be performed while simultaneously ensuring reliable parameter estimation. Model-Based Design of Experiments (DoE)\cite{pukelsheim2006optimal} introduces statistical methods which pose and solve appropriate optimization problems that strategically select experiments from a large set of possible experiments. These selection problems are also common in other contexts, for instance, the application areas that use learning tasks such as active learning\cite{settles2012active}, multi-arm bandits\cite{deshpande2012linear}, diversity sampling\cite{kulesza2012determinantal} require effective selection of input data. 

\begin{figure}[tb]
\centering
\includegraphics[scale=0.1,width=0.5\columnwidth]{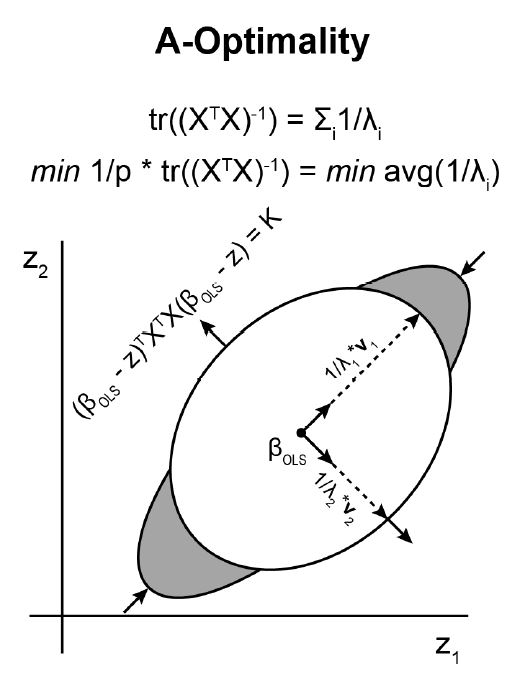}
\caption{Illustrates effect of A-optimality criteria on decreasing variance in estimation of $\beta$, considering a linear model of the form $Y=X\beta+\epsilon$. This variance in estimation is given by $\Sigma^{-1}=(X^\top   X)^{-1}$. In particular, A-optimality minimizes the average inverse singular value($\lambda$). The arrows indicate directions of shrinkage of the ellipse as the criteria is optimized.}
\label{fig: aoptim}
\end{figure}

Linear models of the form $y=f^\top(\z)\beta$, where the measured output depends linearly on the parameters $\beta$ (not necessarily linear on the process variables $\z$) are very common in many of the above applications. DoE for these linear models is critical since they are extensively used, are amenable to simpler interpretations, and offer tractable solution to the experiment selection problem \cite{montgomery2021introduction},\cite{jiao2020does},\cite{lu2017application}. The underlying principle for these statistical techniques is to select a small subset from a pre-specified set $\Omega$ of process variables  $\{\z_i\}_{i=1}^n$ at which the experiments need to be conducted to obtain a reasonable estimate $\hat \beta$ of $\beta\in\reals^p$ such that $f^T(\z)\hat\beta$ is a good (low variance) estimation of $y$ for all $\z$ in $\Omega$.  More precisely, a case with $r$ experimental observations $\y\in \mathbb{R}^r$ at process values $\{\z_i\}_1^r$ ($r\ll n$)  can be written as $\y=X\beta+\epsilon$, where $X\in\reals^{r\times p}$, $f(\z_i)^\top$ is the $i$th row of $X$, and $\epsilon$ is the error vector (typically assumed to be realizations of independent and identically distributed (iid) random variables). 

A least square estimate for this regression problem is given by $\hat\beta=(X^\top    X)^{-1}X^\top\y$ and the variance $var(\hat\beta)=\sigma^2(X^\top    X)^{-1}$, where $\sigma^2$ is the variance of $\epsilon$.  Also the predicted response is $\hat \y(\z)=f^\top(\z)\hat\beta$ with variance $var(\hat \y(\z))=\sigma^2f^\top   (\z)(X^\top X)^{-1}f(\z)$. The DoE problem deals with choosing an appropriate choice of  $r$ process values $\{\z_i\}_{i=1}^r$ from a larger  set $\Omega$, such that the corresponding Fisher information matrix $\Sigma=X^\top    X$ \cite{rissanen1996fisher}  ensures small variance in predicted response; here $r$ is known a priori and is small enough to be cost and time effective. 

In particular, various optimality criteria $g(\Sigma)\in \mathbb{R}$ are used for DoE \cite{pukelsheim2006optimal};  for instance $A$-optimality criterion requires parsimonious selection of ${\z_i}$ such that $g(\Sigma)=\text{trace}(\Sigma^{-1})$ is minimized (See Figure \ref{fig: aoptim}). Other such criteria include D-optimality (minimize $\text{determiniant}(\Sigma^{-1}$), T-optimality (maximize $\trace(\Sigma)$), E-optimality (minimize ${\text{maximum eigenvalue}(\Sigma)}$), V-optimality (minimize average prediction variance),  and G-optimality (minimize worst possible prediction variance) \cite{pukelsheim2006optimal}. These optimization  problems are combinatorial in nature and are known to be NP hard\cite{civril2009selecting}.

These problems become even more complex when  not all process variable values are  preset but need to be chosen or designed. More precisely, in the context of linear regression models considered above, when some elements in the matrix $X$ are themselves parameters to be designed (chosen typically from a continuous domain set). These problems arise in DoE applications where not all process variables $\z_i\in \Omega$ are  determined  a priori, but  are free to be chosen. For instance if $y$ represents taste-quality score of a certain wine and $f_k(\z)$ represents a specific feature (among many other features)  – say that depends on temperature of the room  at which wine is served; the temperature values can be treated as a design variable. Similar objectives also arise in certain missing-data problems, where values for $f_k(\z_i)$ (for some ($k,i$)) are missing and these values need to be imputed.  Missing of data can occur due to various reasons such as unreliable sensors\cite{yick2008wireless}, system malfunction\cite{lo2011progressive}\cite{fletcher2004estimation} or incomplete surveys\cite{raghunathan2004we}. While range of missing data points are known, exact values might still be missing. This is widespread in many research areas and has adverse affects on final model estimation\cite{kromrey1994nonrandomly,ayilara2019impact,marlin2008missing}. The corresponding complexity of DoE problem increases since solving for the above optimality criteria will require solving {\em  simultaneously} the selection problem as well as designing these parameters.

There are many approaches that address individually the optimal DoE problem and missing data imputations but there are none to our knowledge that solve them simultaneously.  Methods specifically developed for DOE that apply to various optimality criteria include  heuristic algorithms such as genetic algorithm \cite{exchangealgo},\cite{heredia2003genetic} and  Fedorov’s exchange\cite{miller1994algorithm} or convex relaxations \cite{pmlr-v70-allen-zhu17e};  some approaches provide better analytical guarantees for specific criteria such as    $T-$ and $D$-optimality \cite{pmlr-v48-ravi16,5671202,li2017polynomial }, or   A-optimality \cite{avron2013faster,JMLR:v18:17-175,nikolov2019proportional,derezinski2018leveraged}. On the other hand imputing values to missing-data problems are generally addressed by using surrogates such as  mean or mode of the data \cite{eekhout2014missing,zhang2016missing,malhotra1987analyzing}; or estimating missing values using maximum likelihood methods\cite{dempstermaximal}. There is no existing literature on the coupled selection-imputation problems.  These problems are nonconvex typically with multiple local minima. Simulations show that standard optimization techniques, such as  interior points method, gets stuck at poor local minima.   Using exchange algorithms\cite{nguyen1992review} to solve this coupled objective becomes computationally intensive as they need to rely on discretized values for continuous variables  and the search space multiplies for every missing value. 

In this work, we pose a combined objective function optimizing for both designable (missing) feature values and experiment selection. In this article, we use $A$-optimality, that is,  $\trace(\Sigma^{-1})$ as the optimality criterion.  We design a flexible solution approach, which in addition to solving the combined objective, can easily incorporate a wide range of constraints on the decision variables. The main heuristic in our solution is casting the problem from the viewpoint of maximum-entropy-principle (MEP) \cite{jaynes1957information,rose1998deterministic, dubey2018maximum,gehler2007deterministic,rao1997design, srivastava2020simultaneous,gull1984maximum, baranwal2019multiway,banerjee2007generalized,xu2014aggregation}. In this viewpoint, we ascribe a probability distribution on the space of combinatorial selections, and determine expected cost function parameterized by design (imputable) values. We  determine this probability distribution  and design values in an iterative process, where a distribution with maximum entropy is computed at each iteration as an upper bound on the expected cost function, is successively decreased. The imputable values at each iteration are simultaneously obtained by minimizing the corresponding unconstrained Lagrangian.   
Our primary contributions in this work include,
\begin{itemize}
    \item Developing a framework for Simultaneous Selection-Imputation Optimization (SSIO) problem.
    \item Posing this problem in MEP framework and developing an iterative algorithm to solve it.
    \item Demonstrating that the iterations in our algorithm mimic a descent method and hence converge to a local minimum.
\end{itemize} 
We also show that our proposed MEP based approach can be adapted to various other linear (non-linear) design optimality criteria. Further, we show a systematic procedure to incorporate application-specific constraints (such as resource capacity constraint) into the SSIO problem. 

Our simulations show empirical success of the developed algorithm. We test our algorithm on multiple randomly generated data. We show over $3.2$ times improvement in comparison to (the standard) mean data imputation for missing data followed by the Fedorov's algorithm \cite{miller1994algorithm} for experiment selection, and over $4.7$ times improvement in comparison to mean data imputation followed by random sampling of experiments \cite{pmlr-v70-allen-zhu17e}. Further, we expound the benefits of {\em annealing} (an integral part of our proposed algorithm described in Section \ref{section: Solution Approach}) by comparing it with the standard optimization techniques (such as interior-point, and Trust-region reflective algorithm \cite{byrd1999interior}) that simultaneously solve for missing data imputation, and experiment selection. In particular, here we show that our algorithm results into a lower cost function value that is $0.57$ times the cost from the above direct methods. The paper is organised as follows, Section \ref{section: Problem Formulation} will deal with problem formulation and modifying it for MEP framework, Section \ref{section: Solution Approach} deals with developing an iterative algorithm to solve for this objective, this is followed by simulations in Section \ref{Section: Simulation} and discussion in Section \ref{section: discussion and analysis}.

\section{Problem Formulation}\label{section: Problem Formulation}

We consider a linear model $y=f^\top   (\z)\beta+\epsilon$, where  $\z$ takes values from a set $\Omega=\{\z_i\}_{i=1}^n$, and $\beta\in\reals^p$ is the parameter vector. We define $X\in\reals^{n\times p}$ such that its $k$th row $\mathbf{x}_k^\top =f^\top (\z_k)$. A model-based $A$-optimality DoE problem requires forming a matrix $X_s\in\reals^{r\times p}$ by appropriately selecting $r\ll n$ rows (experiments) in $X$ such that $\trace\left(\Sigma_s^{-1}\right)$ is minimized, where $\Sigma_s=[X_s^\top    X_s]$ is the Fisher Information matrix. The space of such combinatorial selections are described by the set $S\subset \{0,1\}^n$, where every $\s=\{s_i\}_{i=1}^n\in S$ satisfies $\sum_i s_i=r$, that is
\begin{align}
    S=\{\mathbf{s}:\mathbf{s}=\{s_i\}_{i=1}^n,s_i\in\{0,1\},\sum_i s_i=r\}\label{eq: setS}.
\end{align}
To each selection vector $\s$, we can determine a matrix $X_s\in\reals^{r\times p}$ by deleting all $j$th rows for which $s_j=0$. The corresponding Fisher information matrix can be written as 
\[\Sigma_s=X_s^\top    X_s=X^\top    \diag(\s) X,\] where the $\diag(s)\in \{0,1\}^{n\times n}$ is a diagonal matrix   with $\s$ as its diagonal.

As illustrated in Section \ref{section: introduction}, in this work, we consider an additional objective of imputing feature data while solving A-optimality criterion.  Specifically, let $\textbf{m}\in M$ denote designable or missing feature entries in matrix $X$. Here 
\begin{align}
M=\{\m:\ \m=\text{vec}\left(x_{ij}\right);x_{ij}\in[a_{ij},b_{ij}]\forall(i,j)\in\mathcal{G}\}\label{eq: setM},
\end{align}
and $\mathcal{G}$ specifies the locations of missing data in matrix $X$. 
We denote the input data matrix with missing values $\textbf{m}\in\mathbb{R}^{|\mathcal{G}|}$ by $X(\textbf{m})\in\mathbb{R}^{n\times p}$. Consistent with related literature \cite{nikolov2019proportional,derezinski2018leveraged,pmlr-v70-allen-zhu17e} in linear design of experiments, we assume $X(\m)$ to be of full rank (equal to $p$) and hence $X(\m)^\top X(\m)$ is invertible for all $\m$. The\SSIO problem is given by 

\begin{align}
\label{Eq: original objective}
    \min_{\textbf{s}\in S,\textbf{m}\in M}\quad &\trace\left(\left[X(\textbf{m})^\top   \diag(\textbf{s})X(\textbf{m})\right]^{-1}\right).
\end{align}
 To ensure the invertibility of $(X(\textbf{m})^\top   \diag(\textbf{s})X(\textbf{m}))$, a necessary condition is that the number of selected rows must at least be equal to the number of features in each experiment $i.e\quad r\geq p$. 

In our framework we reformulate the optimization problem (\ref{Eq: original objective}) by changing the decision variable space. We introduce a binary-valued probability distribution $\eta:S\rightarrow \{0,1\}$ on the space of selections  such that $\sum_S\eta(\textbf{s})=1$. Specifically, it is a distribution which takes a unit value at a particular selection of experiments i.e at a vector $\s\in S$ and zero otherwise. Hence the equivalent optimization problem is given by
\begin{align}
    &\min_{\eta,\mathbf{m}\in M}\quad\text{trace}\Big(\Big[\sum_{\mathbf{s}\in S}\eta(\mathbf{s})\left(X(\mathbf{m})^\top   \diag(\mathbf{s})X(\mathbf{m})\right)\Big]^{-1}\Big)\label{eq: etaoptimization}\\
    &\text{subject to}\ 
    \eta(\mathbf{s})\in\{0,1\}\quad \quad\sum_S\eta(\mathbf{s})=1.\label{eta constraint}
\end{align}
In our MEP based framework we relax this binary decision variable $\eta(\textbf{s})$, and replace it with association weight $p(\textbf{s})\in[0,1]$ such that $\sum_Sp(\textbf{s})=1$. The resulting relaxed cost function $\mathcal{A}$ is given by
 \begin{align}\label{eq: final objective}
 \begin{split}
     \mathcal{A}&=\trace\Big(\big(\sum_S p(\textbf{s})X(\textbf{m})^\top   \diag(\textbf{s})X(\textbf{m})\big)^{-1}\Big).\\
     &=\trace\Big(\big(\underbrace{\sum_S p(\textbf{s})\sum_{i=1}^ns_ix_i(\m)x_i(\m)^T}_{D}\big)^{-1}\Big),
 \end{split}
 \end{align}
where $x_i(\m)$ denote the $i$-th row of the $X(\m)$ data matrix. For simplicity of notation we denote $x_i(\m)$ by $x_i$ wherever clear from the context. Subsequently, we use MEP to design these association weights $p(\s)$ as illustrated in the next section.

\section{Solution approach}\label{section: Solution Approach}
In the proposed approach instead of directly solving for (\ref{eq: etaoptimization}) we use MEP based algorithm to solve for the distribution $p(\textbf{s})$ which forms a part of the relaxed cost function (\ref{eq: final objective}). In particular, we solve for the distribution which has maximum entropy while the relaxed cost function $\mathcal{A}$ in (\ref{eq: final objective}) attains a given value $a_0>0$. Intuitively, given a prior information ($\mathcal{A}=a_0$), MEP determines the \textit{most unbiased} distribution and hence maximizes Shannon Entropy $H$. We pose the following optimization problem,
\begin{align}\begin{split}\label{eq:MEP}
 \max_{p(\textbf{s}),\textbf{m}\in M}H&=-\sum_Sp(\textbf{s})\log p(\textbf{s})\\
    \text{subject to}& \quad\mathcal{A}=a_0\\
    &p(\s)\geq 0,
    \  \text{for all}\ \s\in S \ \text{and }\  
   \sum_S p(\s)=1,
    \end{split}
\end{align}
The corresponding Lagrangian to be minimized is  given by  
\begin{align}\label{eq: initial lag}
    L=\mathcal{A}-a_0-TH,
\end{align}
where $T$ denotes Lagrange multiplier. In our framework, we repeatedly solve the above optimization problem at successively decreasing values of $a_0$; as illustrated shortly. 

Note that the set $S$ is combinatorially large (i.e., {\small$|S|=\binom{n}{r}$}). Thus our decision variable space $p(\textbf{s})$ becomes intractable for large values of $n\choose r$. We reduce the search space by incorporating an assumption that $\{s_i\}$ are independent variables which introduces an additional term into the Lagrangian $L$ as shown later. More precisely, we dissociate the distribution $p(\cdot)$ over combinatorially large space $S$ into $p_i(\cdot)$ over individual entries $s_i\in\textbf{s}$, as given by
\begin{align}\label{eq: dissaociate probabilities}
    p(\textbf{s})&=p(s_{1},s_2,\hdots,s_n)=\Pi_i^np_i(s_i).
\end{align}
This assumption of independence  reduces the number of decision variables from $n\choose r$ to $n$ variables as shown below. 
Substituting the above dissociation (\ref{eq: dissaociate probabilities}) into the expression $D$ in the relaxed cost (\ref{eq: final objective}), we obtain
\begin{align}\label{eq:indepD}
     \sum_{\textbf{s}\in S}p(\textbf{s})(\sum_{i=1}^N s_ix_ix_i^\top   )&=\sum_{\textbf{s}\in S}\sum_{i}\prod_{j=1}^n p_j(s_j)s_ix_ix_i^\top   \\
&=\sum_{i=1}^n\sum_{s_i}p_i(s_i)s_ix_ix_i^\top   
\end{align}
Similarly, using (\ref{eq: dissaociate probabilities}) to modify $H$ in (\ref{eq:MEP}), we obtain
\begin{align}\label{eq:indepH}
    H &= - \sum_S p(S)\log p(S)= - \sum_S p(S) \log \prod_{i}p_{i}(s_{i})\\
&= - \sum_{i} \sum_{s_{i}}p_{i}(s_{i})\log p_{i}(s_{i})
\end{align}

For simpler notation, we introduce and define a new variable $q_i:=p_i(s_i=1)$ for each $1\leq i\leq n$; This definition implies  $p_i(s_i=0)=1-q_i$ for each $i$. Therefore,
using (\ref{eq:indepD}) and (\ref{eq:indepH}) in (\ref{eq: initial lag}),  the Lagrangian $L(\mathbf{q},\mu,\m)$ for the missing data A-optimality problem is given by
\begin{align}\label{eq: final lag}
\begin{split}
    L(\mathbf{q},\mu,\m)&=\trace(\sum_i^n q_ix_ix_i^\top)^{-1}-a_0+T\sum_{i} \big[q_{i}\log q_{i}\big]\\ &+ T\big[(1-q_{i})\log (1-q_{i})\big]
    +\mu(\sum_iq_i-r),
    \end{split}
\end{align}
where the last term of this Lagrangian is to account for the constraint of selecting $r$ rows from $n$. 

In equation (\ref{eq: final lag}), L is referred to as free energy and T as temperature borrowing analogies from statistical physics which defines free energy as difference between enthalpy and temperature times entropy. At large values of temperature T, this free energy optimization reduces to minimizing a convex function and hence results in global minimum. As T is annealed (decreased in value), higher weight is given to the relaxed cost function $\mathcal{A}$ changing the Lagrangian from a convex to non-convex function. Finally, as $T\rightarrow 0$, the objective reduces to solving for A-optimality.

We solve for $q_i$ which minimizes $L$ by setting $\frac{\partial L}{\partial q_i}=0$. This results in the following update for $q_i$ \begin{align}
     q_i = \frac{1}{1+\exp\Big\{-\frac{1}{T}x_i^\top R^{-2}x_i + \frac{1}{T}\mu\Big\}},\label{eq: q update}
\end{align}
where $R=\sum_iq_ix_ix_i^\top $. Similarly, to solve for $\mu$, we exploit the fact that $\sum_iq_i=r$, and use $q_i$ from (\ref{eq: q update}), which results in the update of the form
\begin{align}\begin{split}
    \mu&=T\log(\frac{K}{r}),\\ \text{where }\quad K&=\sum_i\frac{1}{\exp\{-\frac{1}{T}\mu\}+\exp\{-\frac{1}{T}x_i^\top R^{-2}x_i\}}.\end{split}\label{eq: mu update}
\end{align}
Finally, we determine the missing values $\m$ (when unconstrained) by setting $\frac{\partial L}{\partial x_{jk}}=0$, where $x_{jk}$ are the entries of the vector $\m\in M$. In particular, 
\begin{align}
    \frac{\partial L}{\partial x_{jk}}&=x_j^\top R^{-2}e_k,\label{eq: xstep}\\
   \frac{\partial L}{\partial x_{jk}}=0&\implies x_{jk}=-\frac{1}{e_k^TR^{-2}e_k}\sum_{l\neq k}x_{jl}e_l^\top R^{-2}e_k,\label{eq: m_update}
\end{align}
where $R=\sum_i q_ix_ix_i^T$, and $e_k\in\mathbb{R}^p$ is a basis vector with value $1$ at the $k$-th location, and $0$ otherwise. It can be seen that the updates for $\textbf{q},\mu$ and $\textbf{m}=\text{vec}(x_{ij}=[X]_{ij}:(i,j)\in\mathcal{G})$ in (\ref{eq: q update}), (\ref{eq: mu update}), and (\ref{eq: m_update}), respectively, are implicit and depend on each other. Hence, in our algorithm we perform the above iterates simultaneously at each value of the annealing parameter $T$. In fact, through the following theorem we show that our iterates in (\ref{eq: q update}), (\ref{eq: mu update}), and (\ref{eq: m_update}) mimic a descent method; and thus, guarantee convergence to a local minimum with appropriate step sizes.
\begin{theorem}
(a) The implicit equation $q_i$ in (\ref{eq: q update}) corresponds to gradient descent step in the auxiliary variable $\xi_i:=-\log\frac{q_i}{1-q_i}$, or equivalently, $q_i = \frac{e^{-\xi_i}}{1+e^{-\xi}}$, where the descent step is given by
\begin{align}
\xi_i^+=\xi_i-T(\exp(\xi_i/2)+\exp(-\xi_i/2))^2\frac{\partial L}{\partial \xi_i},
\end{align}
\\
(b) The implicit equation $\mu$ in (\ref{eq: mu update}) is analogous to the gradient descent step
\begin{align}
    \mu^+=\mu-\frac{T}{\Bar{k}}\frac{\partial L}{\partial \mu},
\end{align}
where $\bar{k}>0$ lies between $\sum_i q_i$ and $r$.
\end{theorem}
\begin{proof}
See Appendix \ref{app: First}.
\end{proof}
Thus, at each given T, the updates $\textbf{q},\mu$ and $\textbf{m}$ given by (\ref{eq: q update}), (\ref{eq: mu update}) and (\ref{eq: m_update}) respectively, are analogous to the gradient descent step of the form
\begin{align}
\begin{bmatrix}
[\xi_i^+]\\
\mu^+\\
[x_{jk}^+]
\end{bmatrix}=
\begin{bmatrix}
[\xi_i]\\
\mu\\
[x_{jk}]
\end{bmatrix}-
\begin{bmatrix}
[\gamma_i] & 0 & 0\\
0 & \zeta & 0\\
0 & 0 & [\phi_{jk}]
\end{bmatrix}
\begin{bmatrix}
[\frac{\partial L}{\partial \xi_i}]\\
\frac{\partial L}{\partial \mu}\\
[\frac{\partial L}{\partial x_{jk}}]
\end{bmatrix},
\end{align}
where $\gamma_i=T(\exp(\xi_i/2)+\exp(-\xi_i/2))^2$, $\zeta=\frac{T}{\bar{k}}$, and $\phi_{jk}>0$ denotes the step size.

Hence, local minimum at each temperature $T$ is obtained by solving (\ref{eq: q update}) for $q_i$, (\ref{eq: mu update}) for $\mu$, and (\ref{eq: m_update}) for $\m$. However, in the case where the missing data $\m$ are constrained to a domain, we solve for $\textbf{m}$ using standard constrained optimization routines like interior points method \cite{byrd1999interior}. We summarize the above steps in Algorithm \ref{alg: Algorithm1}. As stated before, in our algorithm we anneal the temperature from a large value $T\rightarrow \infty$, where the Lagrangian $L$ is dominated by the negative of Shannon entropy $-H$, and the algorithm results into uniform distribution $q_i=r/n$. As $T$ decreases, more weight is given to the cost function $\mathcal{A}$, and the $q_i$'s are no longer uniform. As $T\rightarrow 0$, the Lagrangian $L$ is dominated by the cost function $\mathcal{A}$, distribution $q_i\rightarrow\{0,1\}$ as can be seen from (\ref{eq: q update}), and we minimize the original objective in (\ref{Eq: original objective}).

\begin{algorithm}
{\textbf{Input:} $X\in\mathbb{R}^{n\times p},T_{\text{init}},\alpha<1$, $\{a_{ij}\}$, $\{b_{ij}\}$, $r$, $\m_{\text{init}}$, $T_{\min}$,$T_{\text{init}}$;\\  
\textbf{Output: }{$\textbf{q}$, $\textbf{m}$}\\
$T\leftarrow T_{\text{init}},\textbf{q}\leftarrow\frac{r}{n},\textbf{m}\leftarrow \m_{\text{init}}$\\
\While{$T>T_{\min}$}{
\While{$\text{until convergence}$}{Update $\mathbf{q},\mu$ as in (\ref{eq: q update}), (\ref{eq: mu update}),\\
Update $\m$ using (\ref{eq: m_update}) (unconstrained), or interior-points methods (constrained case).
}
T=$\alpha$T}
\Return ($\textbf{q}$, $\textbf{m}$)
\caption{Design of experiments and imputation of missing data}\label{alg: Algorithm1}}
\end{algorithm}

It can be seen that direct minimization of the A-optimality objective alone by setting the first derivative of the objective function to zero might be stuck at any of multiple local minima and is tightly dependent on initialization. The underlying idea of our proposed algorithm is to find global minimum at large temperatures T, where the problem is convex and track it as we decrease T. As we complete the annealing, the distribution $q_i$ become {\em hard}, and we minimize the original cost function in (\ref{Eq: original objective}). In the subsequent section we compare our Algorithm \ref{alg: Algorithm1} with sequential benchmark methods to demonstrate the efficacy of our proposed methodology.

\section{SIMULATION}\label{Section: Simulation}
In this section we demonstrate the efficacy of our proposed Algorithm \ref{alg: Algorithm1} in solving the simultaneous experiment selection, and feature data imputation problems underlying the model-based DoE. Algorithm \ref{alg: Algorithm1} takes in $X(\textbf{m})\in\mathbb{R}^{n\times  p}$ and $r$ as inputs, where $\textbf{m}$ (as in (\ref{eq: setM})) denotes the missing data in the matrix $X(\textbf{m})$, and $r$ denotes the number of experiments to be selected. Since there are no existing works which simultaneously address missing data imputation and selection of experiments, we provide comparison of our Algorithm \ref{alg: Algorithm1} with a sequential methodology to solve the above two problems. That is, we (a) first implement a commonly used mean imputation for missing feature data $\mathbf{m}$ \cite{eekhout2014missing}, and (b) subsequently use either simple uniform sampling \cite{pmlr-v70-allen-zhu17e} (with equal probability to sample each row) or standard Federov's exchange \cite{miller1994algorithm} algorithm (which swaps selected rows, one at a time, with the ones available to minimize the objective) for selecting the appropriate set of $r$ experiments. 

\begin{figure}
\centering
\includegraphics[width=0.5\columnwidth]{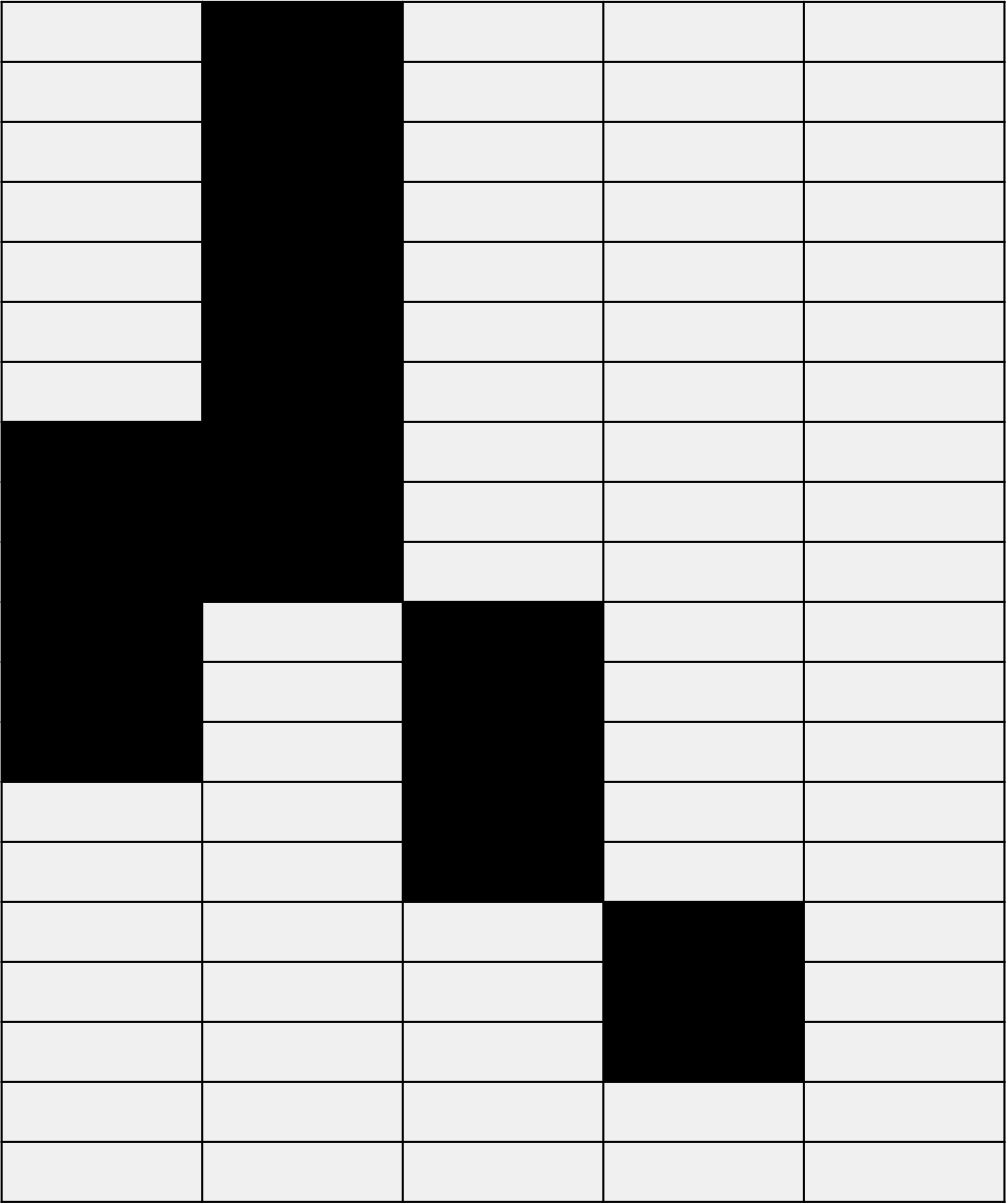}
\caption{Illustrates an example input data matrix $X(\textbf{m})$ with missing data. The darker shade positions indicate missing values and lighter shade denotes a known value in a given range obtained from a random distribution. This example $X\in\mathbb{R}^{20\times 5}$ has $24\%$ missing data.}
\label{fig: missing data}
\end{figure}

In our simulations, we generate incomplete matrices $X(\mathbf{m})$ of various sizes with values missing at randomly selected locations $\mathbf{m}$. Figure \ref{fig: missing data} illustrates an example of such matrices. In the figure, cells marked in darker shade indicate locations of missing values, and those marked in lighter shade denote known entries of the matrix. We evaluate the performance of our Algorithm \ref{alg: Algorithm1}, and compare with the above sequential methodologies on 6 example scenarios. First example E1 generates $X\in\mathbb{R}^{20\times4}$ with $12.5\%$ entries missing from randomly selected locations, and each {\em known} data point in the range of $[-1,2]$. The task is to simultaneously impute for these missing values and select 11 rows (experiments) from $X$. In the second example E2, we create an input data matrix of same size, i.e $X\in\mathbb{R}^{20\times4}$, but with values in range of $[0,4]$ and $10\%$ data missing at random places. The objective here in addition to imputing missing data is to select 12 experiments from the set. Similar examples, E3-E6, are generated to simulate various instances of missing data. See Table \ref{tab:experiments} for details.

\begin{table}[h!]
    \centering
    \begin{tabular}{ |c| c| c| c| c|}
\hline
 No. &$n\times p$ &$\%$ \textbf{m} & r &range of data \\
 \hline
 E1 & $20\times4$ &12.5 &11 &$[-1,2]$\\    
 E2 & $20\times4$ &10 &12 &$[0,4]$\\
 E3 & $20\times4$ &16.25 &12 &$[-2,2]$\\
 E4 & $20\times5$ &24 &11 &$[0,1]$\\
 E5 & $30\times5$ &10 &12 &$[5,10]$\\
 E6 & $30\times5$ &10 &6 &$[5,10]$\\
 \hline
\end{tabular}
        \caption{This table illustrates 6 example simulations we run Algorithm \ref{alg: Algorithm1} and benchmarks on. $n\times p$ indicates size of input matrix $X$, $\%$\textbf{m} indicate the percentage of missing data, r denotes number of rows to be selected, and range of data column indicates the data range of known values. Note that the notation is consistent with ones used in Section \ref{section: Problem Formulation}.}
    \label{tab:experiments}
\end{table}

\begin{figure}
\centering
\includegraphics[width=0.7\columnwidth]{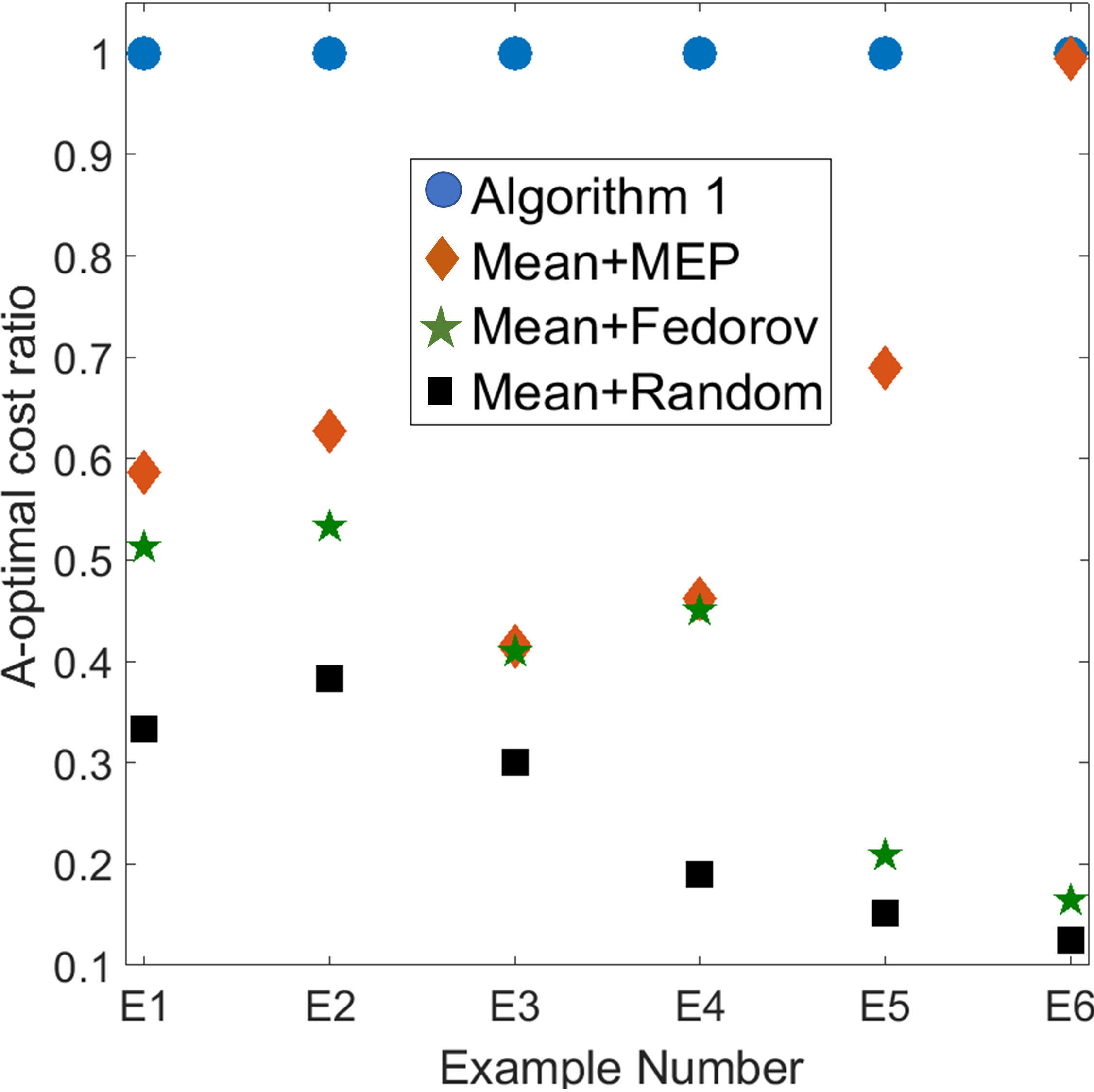}
\caption{Demonstrates the efficacy of Algorithm \ref{alg: Algorithm1} against other benchmarks on examples elaborated in Table \ref{tab:experiments}. For each example on X-axis, we report the ratio of cost obtained from Algorithm \ref{alg: Algorithm1} for solving the combined objective with that obtained from other algorithms on Y-axis (i.e cost as a result of Algorithm \ref{alg: Algorithm1}/cost obtained from benchmark algorithms). As indicated by the legend various shapes of points indicate different benchmarks used. A value less than 1 on Y-axis indicates Algorithm \ref{alg: Algorithm1} results in a cost less than the benchmarks and hence performs better.}
\label{fig: experiments}
\end{figure}

Figure \ref{fig: experiments} illustrates the comparison of the Algorithm \ref{alg: Algorithm1} with the above stated benchmark sequential methodologies. In particular, it plots the ratio of the A-optimal cost incurred using Algorithm \ref{alg: Algorithm1} with the cost incurred when using the sequential methodologies for each of the experiments E1-E6 described in Table \ref{tab:experiments}. Note that the Algorithm \ref{alg: Algorithm1} consistently performs better than the benchmark methods. For instance, in the example E1 our Algorithm \ref{alg: Algorithm1} results into a solution that (a) incurs only $0.6$ times the cost incurred when using mean imputations for missing data followed by our MEP-based algorithm (assuming complete known matrix $X\in \mathbb{R}^{n\times p}$), (b) incurs only $0.52$ times the cost incurred with mean data imputation followed by the standard Fedorovs algorithm, and (c) incurs only $0.35$ times the cost incurred with mean data imputation followed by uniform sampling of experiments. Similarly, in the example E4 our Algorithm \ref{alg: Algorithm1} (a) incurs a cost that is $0.47$ times the cost incurred by mean data imputation followed by our MEP-based algorithm (assuming complete known matrix $X$) to select experiments, (b) incurs a cost that is $0.42$ times the cost incurred when using Fedorovs algorithm for selecting experiments after mean data imputations, and (c) incurs a cost that is $0.2$ times the cost incurred when using uniform sampling for selecting experiments after mean data imputations. Further as seen from Figure \ref{fig: experiments}, given a complete matrix $X$ with imputed values, MEP based Algorithm \ref{alg: Algorithm1} consistently performs better than other standard methods like Fedorov's exchange. Please refer to Figure \ref{fig: experiments} for details on all the other experiments.

To illustrate the benefit of annealing as done in our proposed Algorithm \ref{alg: Algorithm1}, we benchmark our results to that of direct optimization of the posed objective given by (\ref{Eq: original objective}) using standard constrained optimization routine interior points method\cite{byrd1999interior}. On first example, E1, we obtain a cost $0.57$ times lower using Algorithm \ref{alg: Algorithm1} compared to that of of direct optimization. This demonstrates the benefit of annealing in avoiding poor local minimum. An interesting observation while performing the direct optimization using interior points method is that selection of experiments do not change much from the initialization but still, it gives a competitive cost to other algorithms. This is attributed to the imputed missing feature data in a coupled manner and hence strengthens the utility of our proposed problem (\ref{Eq: original objective}).

\section{DISCUSSION AND ANALYSIS}\label{section: discussion and analysis}
Here we highlight some of the features and extensions of our proposed framework. 
\subsection{Generalization to different optimality criteria}
In this work, our MEP-based framework minimizes the A-{\em optimality criteria} in (\ref{Eq: original objective}) to determine the feature data imputations, and select appropriate experiments that ascertain the regression parameter $\beta$. As illustrated in Section \ref{section: introduction}, there are several such linear and non-linear design {\em optimality criteria} developed in literature \cite{pukelsheim2006optimal}. Our proposed MEP-based methodology easily extends to these optimality criteria. For instance, consider the D-optimality design criteria that minimizes the determinant of the variance in estimation of $\beta$ $(\Sigma^{-1})$, i.e. it solves the objective $\min_{\mathbf{s}\in S,\mathbf{m}\in M}\big(\text{det}(X(\mathbf{m})^\top\Lambda(\mathbf{s})X(\mathbf{m})^\top)\big)^{-1/p}$ \cite{pukelsheim2006optimal}. Similar to the case of A-optimality illustrated in the Section \ref{section: Problem Formulation}, here we again reformulate the D-optimaltiy criteria as $\min_{\mathbf{\eta},\mathbf{m}\in M}\big(\text{det}(\sum_{\mathbf{s}\in S}\eta(\s)X(\mathbf{m})^\top\Lambda(\mathbf{\s})X(\mathbf{m})^\top)\big)^{-1/p}$ in terms of the binary decision variable $\eta(\mathbf{s})$. The subsequent solution approach to determine the set of experiments $\mathbf{s}$, and the feature data imputations $\mathbf{m}$ that minimize D-optimality remains similar to the one illustrated in the Section \ref{section: Solution Approach}.

\subsection{Flexibility to incorporate constraints}
Various application areas involving design of experiments require efficient utilization of resources; resulting into several capacity, and feasibility based constraints in the optimization problem (\ref{Eq: original objective}). For instance, consider a scenario where the cost vector $\mathbf{c_i}:=(\mathbf{c_i^1},\hdots,\mathbf{c_i^p})\in \mathbb{R}^p$ indicates the cost (related to individual features) incurred  in performing the $i$-th experiment in $X\in\mathbb{R}^{n\times p}$. The limited budget constraint on the resources pose the constraint of the form $\sum_i\mathbf{c_i}s_i = \mathbf{\kappa}$, where $s_i\in\{0,1\}$ indicates selection of an experiment, and $\mathbf{\kappa}:=(\mathbf{\kappa_1},\hdots,\mathbf{\kappa_p})\in\mathbb{R}^p$ denotes the maximum budget available for the individual features of the experiment. Our framework is flexible, and easily incorporates such constraints. More precisely, we re-interpret the above budget constraint as $\sum_i \mathbf{c_i}q_i = \mathbf{\kappa}$ in terms of our decision variable $q_i:=p_i(s_i=1)$. With this constraint introduced into the optimization problem (\ref{eq:MEP}), the augmented Lagrangian $\mathcal{L}_1$ is given by
\begin{align}\label{eq: augmented_Lag}
\mathcal{L}_1 &= \text{Trace}(\sum_i^n q_ix_ix_i^\top)^{-1}+\mu(\sum_iq_i-r)+T\sum_{i} \big[q_{i}\log q_{i}\big]\nonumber\\ 
&+ T\big[(1-q_{i})\log (1-q_{i})\big]
    + \nu^\top(\sum_i \mathbf{c_i}q_i - \mathbf{\kappa}),
\end{align}
where $\nu\in\mathbb{R}^p$ is the Lagrange parameter. Minimizing $\mathcal{L}_1$ with respect to $q_i$ results into
\begin{align}\label{eq: new_q_const}
&q_i = \frac{1}{1+\exp\Big\{-\frac{1}{T}\big(x_i^\top R^{-2}x_i-\mu-\sum_j \nu_j\mathbf{c_i^j}\big)\Big\}}\\
&~= \frac{\eta_l}{\eta_l+\exp\Big\{-\frac{1}{T}\big(x_i^\top R^{-2}x_i-\mu-\sum_{j}\nu_j\mathbf{c_i^j}+\nu_l\big)\Big\}}\nonumber,
\end{align}
where $\eta_l:=\exp(-\nu_l/T)$. Noting that the desired budget constraint for the $j$-th feature is given by $\sum_i\mathbf{c_i^j}q_i=\mathbf{\kappa_j}$, the update equation for $\eta_l$ is obtained by substituting $q_i$ in the above constraint, i.e., we obtain the update equation
\begin{align}\label{eq: budget_upd}
\eta_l = \frac{\mathbf{\kappa_j}}{\sum_i \mathbf{c_i}^j\frac{1}{\eta_l+\eta_l\exp\Big\{-\frac{1}{T}\big(x_i^\top R^{-2}x_i-\mu-\sum_{j}\nu_j\mathbf{c_i^j}\big)\Big\}}}.
\end{align}
As illustrated in Section \ref{section: Solution Approach}, we can deterministically optimize the Lagrangian $\mathcal{L}_1$ in (\ref{eq: augmented_Lag}) at successively decreasing values of the annealing temperature $T$ by alternating between the equations (\ref{eq: new_q_const}), (\ref{eq: mu update}), and (\ref{eq: budget_upd}) until convergence. {Subsequently, as in Algorithm \ref{alg: Algorithm1} we impute the missing data (if any) $\mathbf{m}$ that minimize (local) $\mathcal{L}_1$ in (\ref{eq: augmented_Lag}) using optimization routines like interior-points method.} As a part of our ongoing research, we are working towards (a) a proof of convergence of the above iterates, and (b) extending our framework to incorporate inequality constraints in the optimization problem (\ref{Eq: original objective}).
\begin{appendices}
\section{}\label{app: First}
{\bf Proof of Theorem 1.} From the Lagrangian L given by (\ref{eq: final lag}), we get
\begin{align}\label{eq: Dlag}
    \frac{\partial L}{\partial q_i}=T\log(\frac{q_i}{1-q_i})+\mu-x_i^\top R^{-2}x_i
\end{align}
Let $\xi_i=-\log(\frac{q_i}{1-q_i})\implies q_i=\frac{\exp(-\xi_i)}{1+\exp(-\xi_i)}$. Hence, from (\ref{eq: Dlag}), the update for $\xi_i$ in each iteration is
\begin{align}
    \xi_i^+=-\frac{1}{T}(x_i^\top R^{-2}x_i-\mu)
\end{align}
Further,
\begin{align}
    \frac{\partial L}{\partial q_i}=\frac{\partial L}{\partial \xi_i}\frac{\partial \xi_i}{\partial q_i}&=-\frac{\partial L}{\partial \xi_i}\big[\frac{1}{q_i}+\frac{1}{1-q_i}\big]\\&=-\frac{\partial L}{\partial \xi_i}\big[\exp(-\xi_i/2)+\exp(\xi_i/2)\big]^2\label{eq: xistep}.
\end{align}
It also follows from (\ref{eq: Dlag}) that
\begin{align}\label{eq: xiDlagq}
    \frac{\partial L}{\partial q_i}=-T\xi_i+T\xi_i^+
\end{align}
From (\ref{eq: xistep}) and (\ref{eq: xiDlagq}),
\begin{align}
    \xi_i^+=\xi_i-T\big[\exp(-\xi_i/2)+\exp(\xi_i/2)\big]^2\frac{\partial L}{\partial \xi_i}\label{eq: xistep}.
\end{align}
Similarly, for $\mu$ using (\ref{eq: mu update}) we directly arrive at
\begin{align}
    \mu^+&=\mu-T\log(\frac{\sum_i q_i}{r})\\
        &=\mu-T\frac{\log(\sum_iq_i)-\log r}{\sum_i q_i-r}(\sum_iq_i-r)\\
        &=\mu-T\frac{\log(\sum_iq_i)-\log r}{\sum_i q_i-r}\frac{\partial L}{\partial \mu}\label{eq: mustep}
\end{align}
Since $f(x)=\log(x)$ satisfies conditions for mean value theorem between $(\sum_i q_i,r)$, there exists $\Bar{k}\in(\sum_i q_i,r)$ such that $\frac{1}{\Bar{k}}=\frac{\log(\sum_iq_i)-\log r}{\sum_i q_i-r}$. Since $\sum_iq_i>0$ and $r>0$, $\Bar{k}>0$.
\end{appendices}
\bibliographystyle{IEEEtran}
\bibliography{IEEEabrv}

\end{document}